\documentclass{article}

\usepackage{marvosym}

\usepackage{amsmath}
\usepackage{amssymb}
\usepackage{amsthm}
\usepackage[T1]{fontenc} 
\usepackage{graphicx}

\usepackage{marginnote}

\newcommand{\intav}[1]{\mathchoice {\mathop{\vrule width 6pt height 3 pt depth -2.5pt\kern -8pt \intop}\nolimits_{\kern -6pt#1}} {\mathop{\vrule width5pt height 3 pt depth -2.6pt \kern -6pt \intop}\nolimits_{#1}}{\mathop{\vrule width 5pt height 3 pt depth -2.6pt \kern -6pt\intop}\nolimits_{#1}} {\mathop{\vrule width 5pt height 3 pt depth-2.6pt \kern -6pt \intop}\nolimits_{#1}}}

\usepackage{ifthen}\usepackage{color}

\usepackage{setspace}

\newtheorem{thm}{Theorem}%

\newtheorem{Definition}{Definition}
\newtheorem{Lemma}{Lemma}
\newtheorem{Corollary}{Corollary}
\newtheorem{Proposition}{Proposition}

\newtheorem{Assumption}{A}

\begin{document}

\title{Potential estimates for fully nonlinear elliptic equations with bounded ingredients}

\author{Edgard A. Pimentel and Miguel Walker}
	\date{\today} 
		
\maketitle

		\begin{center}
			\emph{\small{To Giuseppe Mingione, on the occasion of his 50th birthday,\\with admiration, gratitude and friendship}}
		\end{center}
\vspace{.1in}

\begin{abstract}
\noindent We examine $L^p$-viscosity solutions to fully nonlinear elliptic equations with bounded-measurable ingredients. By considering $p_0<p<d$, we focus on gradient-regularity estimates stemming from nonlinear potentials. We find conditions for local Lipschitz-continuity of the solutions and continuity of the gradient. We briefly survey recent breakthroughs in regularity theory arising from (nonlinear) potential estimates. Our findings follow from -- and are inspired by -- fundamental facts in the theory of $L^p$-viscosity solutions, and results in the work of Panagiota Daskalopoulos, Tuomo Kuusi and Giuseppe Mingione \cite{DKM_2014}. 

\medskip

\medskip 


\end{abstract}

\vspace{.1in}

\section{Introduction}\label{sec_introduction}

We study the regularity of $L^p$-viscosity solutions to
\begin{equation}\label{eq_main}
	F\big(D^2u,Du,u,x\big)=f\hspace{.2in}\mbox{in}\hspace{.2in}\Omega,
\end{equation}
where $F:\mathcal S(d)\times\mathbb{R}^d\times\mathbb{R}\times \Omega\setminus\mathcal{N}\to\mathbb{R}$, is a uniformly elliptic operator with bounded-measurable ingredients, and  $f\in L^p(\Omega)$ for $p>p_0$. Here, $\Omega\subset\mathbb{R}^d$ is an open and bounded domain, $\mathcal{N}$ is a null set, $S(d)\sim\mathbb{R}^\frac{d(d+1)}{2}$ is the space of symmetric matrices, and $d/2<p_0<d$ is the exponent such that the Aleksandrov-Bakelman-Pucci (ABP) estimate is available for elliptic equations with right-hand side in $L^p$, for $p>p_0$.

Our contribution is two-fold. From a mathematical viewpoint, we extend the gradient potential estimates reported in \cite{DKM_2014} to operators with bounded-measurable coefficients depending explicitly on lower-order terms.

We argue by combining well-known facts in the theory of $L^p$-viscosity solutions, obtaining at once the corpus of results in \cite{DKM_2014}. That reasoning leads to the second layer of our contribution: the findings in the paper attest to the broad scope, and consequential character, of the developments reported in \cite{DKM_2014}.

The regularity theory for viscosity solutions to \eqref{eq_main} is a delicate matter. Indeed, the first result in this realm is the so-called Krylov-Sofonov theory. It states that, if $u\in C(B_1)$ is a viscosity solution to
\begin{equation}\label{eq_simple0}
	F(D^2u)\leq0\leq G(D^2u)\hspace{.2in}\mbox{in}\hspace{.2in}B_1
\end{equation}
and $F$ and $G$ are $(\lambda,\Lambda)$-elliptic operators, then $u\in C^{\alpha}_{\rm loc}(B_1)$, for some $\alpha\in(0,1)$ depending only on $d$, $\lambda$ and $\Lambda$. In addition, one derives an estimate of the form
\[
	\left\|u\right\|_{C^{\alpha}(B_{1/2})}\leq C\left\|u\right\|_{L^\infty(B_1)},
\]
where $C=C(d,\lambda,\Lambda)$ \cite{KrSa1}. Indeed, the regularity result in the Krylov-Safonov theory concerns inequalities of the form 
\begin{equation}\label{eq_ksorig}
	a_{ij}(x)\partial^2_{ij}u\leq0\leq b_{ij}(x)\partial^2_{ij}u
\end{equation}
where the matrices $A:=(a_{ij})_{i,j=1}^d$ and $B:=(b_{ij})_{i,j=1}^d$ are uniformly elliptic, \emph{with the same ellipticity constants}. The transition of those inequalities to \eqref{eq_simple0} comes from the fundamental theorem of calculus. Indeed, notice that if $F(0)=G(0)=0$, we get
\[
	\int_0^1\frac{d}{dt}F(tD^2u){\rm d}t=F(D^2u)\leq0\leq G(D^2u)=\int_0^1\frac{d}{dt}G(tD^2u){\rm d}t.
\]
By computing the derivatives above with respect to the variable $t$ and setting
\[
	a_{ij}(x):=\int_0^1D_MF(tD^2u){\rm d}t\hspace{.15in}\mbox{and}\hspace{.15in}b_{ij}(x):=\int_0^1D_MG(tD^2u){\rm d}t,
\]
one notices that a solution to \eqref{eq_simple0} also satisfies \eqref{eq_ksorig}. 

If we replace the inequality in \eqref{eq_simple0} with the equation
\begin{equation}\label{eq_more}
	F(D^2u)=0\hspace{.2in}\mbox{in}\hspace{.2in}B_1
\end{equation}
and require $F$ to be a $(\lambda,\Lambda)$-elliptic operator, solutions become of class $C^{1,\alpha}$ with estimates. Once again, $\alpha\in(0,1)$ depends only on the dimension and the ellipticity \cite{Trudinger,CC}. Finally, if we require $F$ to be uniformly elliptic \emph{and convex} (or concave) viscosity solutions to \eqref{eq_more} are of class $C^{2,\alpha},$ with estimates. This is known as the Evans-Krylov theory, developed independently in the works of Lawrence C. Evans \cite{Evans1} and Nikolai Krylov \cite{Krylov1}.

The analysis of operators with variable coefficients, in the context of non-homogeneous problems first appeared in the work of Luis Caffarelli \cite{Caffarelli_1989}. In that paper, the author considers the equation
\begin{equation}\label{eq_caff}
	F(D^2u,x)=f\hspace{.2in}\mbox{in}\hspace{.2in}B_1
\end{equation}
and requires $F(M,x)$ to be uniformly elliptic. The fundamental breakthrough launched in \cite{Caffarelli_1989} concerns the connection of the variable coefficients operator with its fixed-coefficients counterpart. To be more precise, the author introduces an \emph{oscillation measure} $\beta(x,x_0)$ defined as
\[
	\beta(x,x_0):=\sup_{M\in S(d)}\frac{|F(M,x)-F(M,x_0)|}{1+\|M\|}.
\]
Different smallness conditions on this quantity yield estimates in distinct spaces. It includes estimates in $C^{1,\alpha}$, $W^{2,p}$ and $C^{2,\alpha}$-spaces. Of course, further conditions on the source term $f$ must hold. In particular, it is critical that $f\in L^p(B_1)$, for $p>d$.

An interesting aspect of this theory concerns the continuity hypotheses on the data of the problem. For instance, the regularity estimates do not depend on the continuity of $f$. Meanwhile, the notion of $C$-viscosity solution requires $f$ to be defined everywhere in the domain, as it depends on pointwise inequalities \cite{CL,CEL,CIL}. Hence, asking $f$ to be merely a measurable function in some Lebesgue space is not compatible with the theory. See the last paragraph before Theorem 1 in \cite{Caffarelli_1989}.

In \cite{CCKS_1996}, the authors propose an $L^p$-viscosity theory, recasting the notion of viscosity solutions in an \emph{almost-everywhere} sense. In that paper, the authors examine \eqref{eq_main} and suppose the ingredients of the problem are in $L^p$, for $p>p_0$. The quantity $d/2<p_0<d$ appeared in the work of Eugene Fabes and Daniel Stroock \cite{FS1984}. It stems from the improved integrability of the Green function for $(\lambda,\Lambda)$-linear operators.

In \cite{Escauriaza}, and before the formalization of $L^p$-viscosity solutions, the quantity $p_0$ appeared in the context of Sobolev regularity. In that paper, Luis Escauriaza resorted to the improved integrability of the Green function from \cite{FS1984} to extend Caffarelli's $W^{2,p}$-regularity theory to the range $p_0<p<d$. For that reason, $p_0$ is referred to in the literature as Escauriaza's exponent. 

A fundamental study of the regularity theory for $L^p$-viscosity solutions to \eqref{eq_main} appeared in \cite{Swiech_1997}. Working merely under uniform ellipticity, the author proves regularity results for the gradient of the solutions. In case $p>d$, solutions are of class $C^{1,\alpha}$. Here, the smoothness degree depends on the Krylov-Safonov exponent, and on the ratio $d/p$. However, in case $p_0<p\leq d$, solutions are only in $W^{1,q}$, where $q\to\infty$ as $p\to d$.

The findings in \cite{Swiech_1997} highlight an important aspect of the theory, namely: the smoothness of $Du$, in the range $p_0<p<d$, is a very delicate matter. It is known that $C^{1,\alpha}$-regularity is not available in this context. 

A program that successfully accessed this class of information is the one in \cite{DKM_2014}. Through a modification in the linear Riesz potential, tailored to accommodate the $p$-integrability of the data, the authors produce potential estimates for the $L^p$-viscosity solutions to \eqref{eq_caff}. Ultimately, those estimates yield a modulus of continuity for the gradient of the solutions.

In addition to uniform ellipticity, the results in \cite{DKM_2014} require an average control on the oscillation of $F(M,x)$. It also assumes $f\in L^p(\Omega)$ for $p_0<p<d$. Under these conditions, the authors prove a series of potential estimates. Those lead to local boundedness and (an explicit modulus of) continuity for $Du$. Also, a borderline condition in Lorentz spaces follows: if $f\in L^{d,1}(\Omega)$, then $Du$ is continuous. Besides providing new, fundamental developments to the regularity theory of fully nonlinear elliptic equations, the arguments in \cite{DKM_2014} are pioneering in taking to the non-variational setting a class of methods available before only for problems in the divergence form.

We extend the findings in \cite{DKM_2014} to the case of \eqref{eq_main} in the presence of bounded-measurable ingredients. Our analysis heavily relies on properties of $L^p$-viscosity solutions \cite{CCKS_1996,Swiech_1997}; see also \cite{Winter_2009}.

Our first main result concerns the Lipschitz-continuity of $L^p$-viscosity solutions to \eqref{eq_main} and reads as follows.

\begin{thm}[Lipschitz continuity]\label{thm_main1}
Let $u\in C(\Omega)$ be an $L^p$-viscosity solution to \eqref{eq_main}. Suppose Assumptions A\ref{assump_F} and A\ref{assump_f} are in force. Then, for every $q>d$, there exists a constant $\theta^*=\theta^*(d,\lambda,\Lambda,p,q)$ such that if Assumption A\ref{assump_osc} holds with $\theta\equiv\theta^*$, one has
\[
	\left|Du(x)\right|\leq C\left[{\bf I}^f_p(x,r)+\left(\intav{B_r(x)}\left|Du(y)\right|^q{\rm d}y\right)^\frac{1}{q}\right]
\]
for every $x\in \Omega$ and $r>0$ with $B_r(x)\subset \Omega$, for some universal constant $C>0$.
\end{thm}

The potential estimate in Theorem \ref{thm_main1} builds upon \'Swi\k{e}ch's $W^{1,q}$-estimates to produce uniform estimates in $B_{1/2}$. In fact, by taking $d<q<p^*$ in Theorem \ref{thm_main1}, with
\[
	p^*:=\frac{pd}{d-p},\hspace{.2in}\mbox{and}\hspace{.2in}d^*=+\infty,
\]
one finds $C=C(d,\lambda,\Lambda,p)$ such that 
\[
	\left\|Du\right\|_{L^\infty(B_{1/2})}\leq C\left(\left\|u\right\|_{L^\infty(B_1)}+\left\|f\right\|_{L^p(B_1)}\right).
\]

Our second main result establishes gradient-continuity for the $L^p$-solutions to \eqref{eq_main} and provides an explicit modulus of continuity for the gradient. It reads as follows.

\begin{thm}[Gradient continuity]\label{thm_main2}
Let $u\in C(\Omega)$ be an $L^p$-viscosity solution to \eqref{eq_main}. Suppose Assumptions A\ref{assump_F} and A\ref{assump_f} are in force. Suppose further that ${\bf I}^f_p(x,r)\to 0$ as $r\to0$, uniformly in $x$. There exists $0<\theta^*\ll1$ such that, if Assumption A\ref{assump_osc} holds for $\theta\equiv\theta^*$, then $Du$ is continuous. In addition, for $\Omega'\Subset\Omega''\Subset \Omega$, and any $\delta\in (0,1]$, one has
\[
	|Du(x)-Du(y)|\leq C\left(\|Du\|_{L^\infty(\Omega')}|x-y|^{\alpha(1-\delta)}+\sup_{x\in\Omega}\;{\bf I}^f_p\left(x,4|x-y|^\delta\right)\right),
\]
for every $x,y\in\Omega'$, where $C=C(d,p,\lambda,\Lambda, \omega,\Omega',\Omega'')$ and $\alpha=\alpha(d,p,\lambda,\Lambda)$.
\end{thm}

The strategy to prove Theorems \ref{thm_main1} and \ref{thm_main2} combines fundamental facts in $L^p$-viscosity theory to show that a solution to \eqref{eq_main} also solves an equation of the form 
\[
	\tilde{F}(D^2u,x)=\tilde f\hspace{.2in}\mbox{in}\hspace{.2in}\Omega,
\]
where $\tilde F$ and $\tilde f$ meet the conditions required in \cite{DKM_2014}. In particular, the Lorentz borderline condition for gradient-continuity follows as a corollary.

\begin{Corollary}[Borderline gradient-regularity]
Let $u\in C(\Omega)$ be an $L^p$-viscosity solution to \eqref{eq_main}. Suppose Assumptions A\ref{assump_F} and A\ref{assump_f} are in force. Suppose further $f\in L^{d,1}(\Omega)$. There exists $0<\theta^*\ll1$ such that, if Assumption A\ref{assump_osc} holds for $\theta\equiv\theta^*$, then $Du$ is continuous.  
\end{Corollary}

We organize the remainder of this paper as follows. Section \ref{sec_potential} presents some context on potential estimates, briefly describing their motivation and mentioning recent breakthroughs. We detail our main assumptions in Section \ref{subsec_ma}, whereas Section \ref{subsec_prelim} gathers preliminary material. The proofs of Theorems \ref{thm_main1} and \ref{thm_main2} are the subject of Section \ref{sec_proof}.

\section{Potential estimates: from the Poisson equation to fully nonlinear problems}\label{sec_potential}

Potential estimates are natural in the context of linear equations for which a representation formula is available. For instance, let $\mu\in L^1(\mathbb{R}^d)$ be a measure and consider the Poisson equation
\begin{equation}\label{eq_poissa}
	-\Delta u=\mu\hspace{.2in}\mbox{in}\hspace{.2in}\mathbb{R}^d.
\end{equation}
It is well-known that $u$ can be represented through the convolution of $\mu$ with the appropriate Green function. In case $d>2$, we have
\begin{equation}\label{eq_rice}
	u(x)=C\int_{\mathbb{R}^d}\frac{\mu(y)}{|x-y|^{d-2}}{\rm d}y,
\end{equation}
where $C>0$ depends only on the dimension. 

Now, recall the $\beta$-Riesz potential of a Borel measure $\mu\in L^1(\mathbb{R}^d)$ is given by
\[
	I_\beta^\mu(x):=\int_{\mathbb{R}^d}\frac{\mu (y)}{|x-y|^{d-\beta}}{\rm d}y.
\]
Hence, the representation formula \eqref{eq_rice} allows us to write $u(x)$ as the $2$-Riesz potential of $\mu$. Immediately one infers that
\[
	|u(x)|\leq C\left|I_2^\mu(x)\right|,
\]
obtaining a \emph{potential} estimate for $u$. By differentiating \eqref{eq_rice} with respect to an arbitrary direction $e\in\mathbb{S}^{d-1}$, one concludes
\[
	|Du(x)|\leq C\left|I_1^\mu(x)\right|.
\]
That is, \emph{the representation formula} available for the solutions to the Poisson equation yields potential estimates for the solutions. 

This reasoning collapses if \eqref{eq_poissa} is replaced with a nonlinear equation lacking representation formulas. Then a \emph{fundamental question} arises: it concerns the availability of potential estimates for (nonlinear and inhomogeneous) problems for which representation formulas are not available.

The first answer to that question appears in the works of \textsc{Tero Kilpel\"ainen} and \textsc{Jan Mal\'y} \cite{KM1992}, and \textsc{Neil Trudinger} and \textsc{Xu-Jia Wang} \cite{TW2002}, where the authors produce potential estimates for the solutions of $p$-Poisson type equations. Taking this approach a notch up, and accounting for potential estimates \emph{for the gradient} of solutions, one finds the contributions of \textsc{Giuseppe Mingione} \cite{M2018,M2014,M2011,M2011a},  \textsc{Frank Duzaar} and \textsc{Giuseppe Mingione} \cite{DM2011,DM2010,DM2010a,DM2009}, and \textsc{Tuomo Kuusi} and \textsc{Giuseppe Mingione} \cite{KM2018, KM2016, KM2016a, KM2014, KM2014a, KM2014b, KM2014c, KM2013, KM2013a, KM2012, KM2012a, KM2011}. Of particular interest to the present article is the analysis of potential estimates in the fully nonlinear setting, due to \textsc{Panagiota Daskalopoulos}, \textsc{Tuomo Kuusi}, and \textsc{Giuseppe Mingione} \cite{DKM_2014}. More recent contributions appeared in the works of \textsc{Cristiana De Filippis} \cite{CDF2022} and \textsc{Cristiana De Filippis} and \textsc{Giuseppe Mingione} \cite{FM2021,FM2022}. See also the works of \textsc{Cristiana De Filippis} and collaborators \cite{FS2022,FP2022}.

In \cite{DM2011} the authors examine an equation of the form
\begin{equation}\label{eq_dm11}
	-{\rm div}\,a(x,Du)=\mu\hspace{.2in}\mbox{in}\hspace{.2in}\Omega,
\end{equation}
where $\Omega\subset\mathbb{R}^d$ is a Lipschitz domain, and $\mu\in L^1(\Omega)$ is a Radon measure with finite mass. Here, $a:\Omega\times\mathbb{R}^d\to\mathbb{R}^d$ satisfies natural conditions, regarding growth, ellipticity, and continuity. Those conditions involve an inhomogeneous exponent $p\geq 2$, concerning the behaviour of $a=a(x,z)$ on $z$. An oversimplification yields
\[
	a(x,z)=|z|^{p-2}z,
\]
for $p>2$, turning \eqref{eq_dm11} into the degenerate $p$-Poisson equation. In that paper, the authors resort to the Wolff potential ${\bf W}^\mu_{\beta,p}$, defined as
\[
	{\bf W}^\mu_{\beta,p}(x,R):=\int_0^R\frac{1}{r^\frac{d-\beta p}{p-1}}\left(\int_{B_r(x)}\mu(y){\rm d}y\right)^\frac{1}{p-1}\frac{{\rm d}r}{r},
\]
for $\beta\in(0,d/p]$. Their main result is a pointwise estimate for the gradient of the solutions to \eqref{eq_dm11}. It reads as
\begin{equation}\label{eq_estdm11}
	\left|Du(x)\right|\leq C\left[\intav{B_R(x)}\left|Du(y)\right|{\rm d}y+{\bf W}_{\frac{1}{p},p}^\mu(x,2R)\right],
\end{equation}
whenever $B_R(x)\subset\Omega$, and $R>0$ is bounded from above by some universal quantity depending also on the data of the problem; see \cite[Theorem 1.1]{DM2011}. 
A remarkable consequence of this estimate is a Lipschitz-continuity criterium for $u$ obtained solely in terms of the Wolff potential of $\mu$. Indeed, if ${\bf W}^\mu_{\frac{1}{p},p}(\cdot,R)$ is essentially bounded for some $R>0$, every $W^{1,p}_0$-weak solution to \eqref{eq_dm11} would be locally Lipschitz continuous. 
We notice the nonlinear character of the Wolff potential suits the growth conditions the authors impose on $a(x,z)$, as it scales accordingly under Lipschitz geometries.

The findings in \cite{DM2011} also respect a class of very weak solutions, known as solutions obtained by limit of approximations (SOLA); see \cite{BG1989,BG1992}. This class of solutions is interesting because, among other things, it allows us to consider functions in larger Sobolev spaces. Indeed, for $2-1/d<p<d$ one can prove the existence of a SOLA $u\in W^{1,1}_0(\Omega)$ to 
\[
	\begin{cases}
		-\Delta_p u=\mu&\hspace{.2in}\mbox{in}\hspace{.2in}\Omega\\
		u=0&\hspace{.2in}\mbox{on}\hspace{.2in}\partial\Omega.
	\end{cases}
\]
In addition, $u\in W^{1,q}_0(\Omega)$ with estimates, provided $q>1$ such that
\[
	1<q<\frac{d(p-1)}{d-1}.
\]

When it comes to the proof of \eqref{eq_estdm11}, the arguments in \cite{DM2011} are very involved. However, one notices a fundamental ingredient. Namely, a \emph{decay rate for the excess} of the gradient with respect to its average. Indeed, the authors prove there exist $\beta\in(0,1]$ and $C\geq 1$ such that
\begin{equation}\label{eq_excess}
	\intav{B_r(x)}\left|Du(y)-(Du)_{r,x}\right|{\rm d}y\leq C\left(\frac{r}{R}\right)^\beta\intav{B_R(x)}\left|Du(y)-(Du)_{R,x}\right|{\rm d}y,
\end{equation}
for every $0<r<R$ with $B_R(x)\subset\Omega$. Here,
\[
	(Du)_{\rho,x}:=\intav{B_\rho(x)}Du(z){\rm d}z.
\]
See \cite[Theorem 3.1]{DM2011}. An important step in the proof of \eqref{eq_excess} is a measure alternative, depending on the fraction of the ball $B_r$ in which the gradient is larger than, or smaller than, some radius-dependent quantity.

Although the Wolff potential captures the inhomogeneous and nonlinear aspects of $a=a(x,z)$, a natural question concerns the use of linear potentials in the analysis of \eqref{eq_dm11}.

Indeed, in \cite{M2011} the author supposes $a(x,z)$ to satisfy
\begin{equation}\label{eq_condm11}
	\begin{cases}
		&\lambda|\xi|^2\leq \left\langle \partial_za(x,z)\xi,\xi\right\rangle,\\
		&|\partial_za(x,z)|+|a(x,0)|\leq C,\\
		&|a(x,z)-a(y,z)|\leq K|x-y|^\alpha(1+|z|),		
	\end{cases}	
\end{equation}
for every $x,y\in\Omega$, $z\in \mathbb{R}^d$, and $\xi\in\mathbb{R}^d$, for some $C,\lambda>0$, and $\alpha\in (0,1]$. Under these natural conditions, he derives a gradient bound in terms of the (linear) localized Riesz potential ${\bf I}^\sigma_\beta(x,R)$, defined as
\[
	{\bf I}^\sigma_\beta(x,R):=\int_0^R\frac{1}{r^{d-\beta}}\left(\int_{B_r(x)}\sigma(y){\rm d}y\right)\frac{{\rm d}r}{r},
\]
for a measure $\sigma\in L^1(\Omega)$, and $\beta\in(0,1]$, whenever $B_R(x)\subset\Omega$. 

Indeed, the main contribution in \cite{M2011} is the following: under \eqref{eq_condm11}, solutions to \eqref{eq_dm11} satisfy
\begin{equation}\label{eq_estm11}
	\left|Du(x)\right|\leq C\left[\intav{B_R(x)}|Du(y)|{\rm d}y+{\bf I}_1^\mu(x,2R)+K\left({\bf I}_\alpha^{|Du|}(x,2R)+R^\alpha\right)\right],
\end{equation}
where $C>0$ depends on the data in \eqref{eq_condm11}. In case $a=a(z)$ does not depend on the spatial variable, $K\equiv 0$ and \eqref{eq_estm11} recovers the usual potential estimate.

A further consequence of potential estimates is in unveiling the borderline conditions for $C^1$-regularity of the solutions to \eqref{eq_dm11}. See \cite{DM2010}; see also \cite{C2011} for related results. More precisely, the intrinsic connection between Lorentz spaces and the nonlinear Wolff potentials unlocks the minimal conditions on the right-hand side $\mu$ that ensures continuity of $Du$.

In \cite{DM2010}, the authors impose $p$-growth, ellipticity, and continuity conditions on $a=a(x,z)$, and derive minimal requirements on $\mu$ to ensure that $u\in C^1(\Omega)$ \cite[Theorem 3]{DM2010}; see also \cite[Theorem 9]{DM2010} for the vectorial counterpart of this fact.

They prove that if $\mu\in L^{d,\frac{1}{p-1}}_{\rm loc}(\Omega)$, then $Du$ is continuous in $\Omega$. To get this fact, one first derives an estimate for the Wolff potential ${\bf W}^\mu_{\frac{1}{p},p}(x,R)$ in terms of the $(d,1/(p-1))$-Lorentz norm of $\mu$. It follows from averages of decreasing rearrangements of $\mu$. See \cite[Lemma 2]{DM2010}. Then one notices that such control implies
\[
	{\bf W}^\mu_{\frac{1}{p},p}(x,R)\to 0
\]
uniformly in $x\in \Omega$, as $R\to 0$; see \cite[Lemma 3]{DM2010}. 

The previous (very brief) panorama of the literature suggests that whenever $a=a(x,z)$ satisfies natural conditions -- concerning $p$-growth, ellipticity, and continuity -- potential estimates are available for the solutions to \eqref{eq_dm11}. Those follow through Wolff and (linear) Riesz potentials. Furthermore, this approach comes with a borderline criterion on $\mu$ for the differentiability of solutions. However, these developments appear in the variational setting, closely related to the notion of weak distributional solutions. 

Potential estimates in the non-variational case are the subject of \cite{DKM_2014}. In that paper, the authors examine fully nonlinear elliptic equations
\begin{equation}\label{eq_flbla}
	F(D^2u,x)=f\hspace{.2in}\mbox{in}\hspace{.2in}\Omega,
\end{equation}
where $F$ is uniformly elliptic and $f\in L^p(B_1)$. In this context, the appropriate notion of solution is the one of $L^p$-viscosity solution \cite{CCKS_1996}. Technical aspects of the theory -- including its very definition -- rule out the case where $f\in L^1(\Omega)$, regardless of the dimension $d\geq 2$. Instead, the authors work in the range $p_0<p<d$, where $d/2<p_0<d$ is the exponent associated with the Green's function estimates appearing in \cite{FS1984}.

The consequences of potential estimates for fully nonlinear equations are remarkable. In fact, if $f\in L^p(\Omega)$ with $p>d$, solutions to \eqref{eq_flbla} are known to be of class $C^{1,\alpha}$, with $\alpha\in(0,1)$ satisfying
\[
	\alpha<\min\left\lbrace \alpha_0,1-\frac{d}{p}\right\rbrace,
\]
where $\alpha_0\in(0,1)$ is the exponent in the Krylov-Safonov theory available for $F=0$; see \cite{Swiech_1997}. It is also known that $C^{1,\alpha}$-regularity is no longer available for \eqref{eq_flbla} in case $p<d$. The fundamental question arising in this scenario concerns the \emph{regularity of $Du$ in the Escauriaza range $p_0<p<d$.}

In \cite{Swiech_1997}, the author imposes an oscillation control on $F(M,\cdot)$ with respect to its fixed-coefficients counterpart and proves regularity estimates for the solutions in $W^{1,q}(\Omega)$, for $p_0<p<d$, for every
\[
	q<p^*:=\frac{pd}{d-p},
\]
with $d^*:=+\infty$. Meanwhile, the existence of a gradient in the classical sense, or any further information on its degree of smoothness, was not available in the $p<d$ setting.

In \cite{DKM_2014} the authors consider $L^p$-viscosity solutions to \eqref{eq_flbla}, with $f\in L^p(\Omega)$, for $p_0<p<d$. In this context, they prove the local boundedness of $Du$ in terms of a $p$-variant of the (linear) Riesz potential. In addition, the authors derive continuity of the gradient, with an explicit modulus of continuity. Finally, they obtain a borderline condition on $f$, once again involving Lorentz spaces. In fact, if $f\in L^{d,1}(\Omega)$, then $u\in C^1(\Omega)$.

The reasoning in \cite{DKM_2014} involves the excess of the gradient vis-a-vis its average and a decay rate for this quantity. However, in the context of viscosity solutions, energy estimates are not available as a starting point for the argument. Instead, the authors cleverly resort to \'Swi\k{e}ch's $W^{1,q}$-estimates and prove a decay of the excess at an initial scale. An involved iteration scheme builds upon the natural scaling of the operator and unlocks the main building blocks of the argument.

\section{Technical preliminaries and main assumptions}\label{sec_prelim}

This section details our assumptions and gathers basic notions and facts used throughout the paper. We start by putting forward the former.

\subsection{Main assumptions}\label{subsec_ma} 

For completeness, we proceed by defining the extremal Pucci operators $\mathcal{P}_{\lambda,\Lambda}^\pm:S(d)\to\mathbb{R}$.

\begin{Definition}[Pucci extremal operators]\label{def_pucci}
Let $0<\lambda\leq \Lambda$. For $M\in S(d)$ denote with $\lambda_1,\ldots,\lambda_d$ its eigenvalues. We define the Pucci extremal operator $\mathcal{P}^+_{\lambda,\Lambda}:S(d)\to\mathbb{R}$ as 
\[
	\mathcal{P}^+_{\lambda,\Lambda}(M):=-\lambda\sum_{\lambda_i>0}\lambda_i+\Lambda\sum_{\lambda_i<0}\lambda_i.
\]
Similarly, we define the Pucci extremal operator $\mathcal{P}^-:S(d)\to\mathbb{R}$ as
\[
	\mathcal{P}^-_{\lambda,\Lambda}(M):=-\Lambda\sum_{\lambda_i>0}\lambda_i+\lambda\sum_{\lambda_i<0}\lambda_i.
\]
\end{Definition}

\begin{Assumption}[Structural condition]\label{assump_F}
Let $\omega:[0,+\infty)\to[0,+\infty)$ be a modulus of continuity, and fix $\gamma>0$. We suppose the operator $F$ satisfies 
\[
	\begin{split}
		\mathcal{P}_{\lambda,\Lambda}^-(M-N)-\gamma|p-q|-\omega(|r-s|)&\leq F(M,p,r,x)-F(N,q,s,x)\\
			&\leq\mathcal{P}_{\lambda,\Lambda}^+(M-N)+\gamma|p-q|+\omega(|r-s|),
	\end{split}
\]
for every $(M,p,r)$ and $(N,q,s)$ in $S(d)\times\mathbb{R}^d\times\mathbb{R}$, and every $x\in \Omega\setminus\mathcal{N}$. Also, $F=F(M,p,r,x)$ is non-decreasing in $r$ and $F(0,0,0,x)=0$.
\end{Assumption}

Our next assumption sets the integrability of the right-hand side $f$.

\begin{Assumption}[Integrability of the right-hand side]\label{assump_f}
We suppose $f\in L^p(B_1)$, for $p>p_0$, where $d/2<p_0<d$ is the exponent such that the ABP maximum principle holds for solutions to uniformly elliptic equations $F=f$ provided $f\in L^p$, with $p<p_0$.
\end{Assumption}

We continue with an assumption on the oscillation of $F$ on $x$. To that end, consider
\[
	\beta(x,y):=\sup_{M\in S(d)\setminus\{0\}}\frac{\left|F(M,0,0,x)-F(M,0,0,y)\right|}{\|M\|}.
\]
We proceed with a smallness condition on $\beta(\cdot,y)$, uniformly in $y\in B_1$.

\begin{Assumption}[Oscillation control]\label{assump_osc}
For every $y\in \Omega$, we have
\[
	\sup_{B_r(y)\subset \Omega}\;\intav{B_r(y)}\beta(x,y)^p{\rm d}x\leq \theta^p,
\]
where $0<\theta\ll 1$ is a small parameter we choose further in the paper.
\end{Assumption}

We close this section with a remark on the modulus of continuity $\omega$ appearing in Assumption A\ref{assump_F}. For any $v\in C(B_1)\cap L^\infty(B_1)$ we notice that $\omega(|v(x)|)\leq C$ for some $C>0$, perhaps depending on the $L^\infty$-norm of $v$. Hence 
\[
	\left(\int_{B_1}\omega(|v(x)|)^p{\rm d}x\right)^\frac{1}{p}\leq C.
\]
This information will be useful when estimating certain quantities in $L^p$-spaces appearing further in the paper.

\subsection{Preliminaries}\label{subsec_prelim}

In the sequel, we introduce the basics of $L^p$-viscosity solutions, mainly focusing on the properties we use in our arguments. We start with the definition of $L^p$-viscosity solutions for \eqref{eq_main}.

\begin{Definition}[$L^p$-viscosity solution]\label{def_lpvisc}
Let $F=F(M,p,r,x)$ be nondecreasing in $r$ and $f\in L^p(B_1)$ for $p>d/2$. We say that $u\in C(\Omega)$ is an $L^p$-viscosity subsolution to $F=f$ if for every $\phi\in W^{2,p}_{\rm loc}(\Omega)$, $\varepsilon>0$ and open subset $\mathcal{U}\subset \Omega$ such that 
\[
	F(D^2\phi(x),D\phi(x),u(x),x)-f(x)\geq \varepsilon
\]
almost everywhere in $\mathcal{U}$, then $u-\phi$ cannot have a local maximum in $\mathcal{U}$. We say that $u\in C(\Omega)$ is an $L^p$-viscosity supersolution to $F=f$ if for every $\phi\in W^{2,p}_{\rm loc}(\Omega)$, $\varepsilon>0$ and open subset $\mathcal{U}\subset \Omega$ such that 
\[
	F(D^2\phi(x),D\phi(x),u(x),x)-f(x)\leq -\varepsilon
\]
almost everywhere in $\mathcal{U}$, then $u-\phi$ cannot have a local minimum in $\mathcal{U}$. We say that $u\in C(\Omega)$ is an $L^p$-viscosity solution to $F=f$ if it is both an $L^p$-sub and an $L^p$-supersolution to $F=f$.
\end{Definition}

Although the definition of $L^p$-viscosity solutions requires $p>d/2$, the appropriate range for the integrability of the data is indeed $p>p_0>d/2$, as most results in the theory are available only in this setting. See, for instance, \cite{CCKS_1996}. For further reference, we recall a result on the twice-differentiability of $L^p$-viscosity solutions.

\begin{Lemma}[Twice-differentiability]\label{lem_puredequeijo}
Let $u\in C(\Omega)$ be an $L^p$-viscosity solution to \eqref{eq_main}. Suppose Assumptions A\ref{assump_F} and A\ref{assump_f} are in force. Then $u$ is twice differentiable almost everywhere in $\Omega$. Moreover, its pointwise derivatives satisfy the equation almost everywhere in $\Omega$.
\end{Lemma}

For the proof of Lemma \ref{lem_puredequeijo}, see \cite[Theorem 3.6]{CCKS_1996}. In what follows, we present a lemma relating $L^p$-viscosity solutions to $F=f$ with equations governed by the extremal Pucci operators.

\begin{Lemma}\label{lem_prelvc}
Suppose Assumption A\ref{assump_F} is in force and $f\in L^p(\Omega)$, with $p>p_0$. Suppose further that $u\in C(\Omega)$ is twice differentiable almost everywhere in $\Omega$. Then $u$ is an $L^p$-viscosity subsolution {\rm [resp. supersolution]} of \eqref{eq_main} if and only if
\begin{enumerate}
\item[i.] we have 
\[
	\begin{split}
		F(D^2u(x)&,Du(x),u(x),x)\leq f(x)\\
		{\rm [\mbox{\rm resp. }} F(D^2u(&x),Du(x),u(x),x)\geq f(x){\rm ]}
	\end{split}
\]
almost everywhere in $\Omega$, and
\item[ii.] whenever $\phi\in W^{2,p}_{\rm loc}(\Omega)$ and $u-\phi$ has a local maximum {\rm [resp. minimum]} at $x^*$, then
\[
	\begin{split}
		{\rm ess}&\liminf_{x\to x^*}\left(\mathcal{P}^-\left(D^2(u-\phi)(x)\right)-\gamma\left|D(u-\phi)(x)\right|\right)\geq 0\\
		{\rm [\mbox{\rm resp. }}\hspace{.1in} &{\rm ess}\limsup_{x\to x^*}\left(\mathcal{P}^+\left(D^2(u-\phi)(x)\right)+\gamma\left|D(u-\phi)(x)\right|\right)\leq 0{\rm ]}.
	\end{split}
\]
\end{enumerate}
\end{Lemma}

For the proof of Lemma \ref{lem_prelvc}, we refer the reader to \cite[Lemma 1.5]{Swiech_1997}. We are interested in a consequence of Lemma \ref{lem_prelvc} that allows us to relate the solutions of $F(D^2u,Du,u,x)=f$ with the equation $F(D^2u,0,0,x)=\tilde{f}$, for some $\tilde{f}\in L^p(\Omega)$. This is the content of the next corollary.

\begin{Corollary}\label{cor_salvador}
Let $u\in C(\Omega)$ be an $L^p$-viscosity solution to \eqref{eq_main}. Suppose A\ref{assump_F} and A\ref{assump_f} hold. Define $\tilde{f}:\Omega\to\mathbb{R}$ as 
\[
	\tilde{f}(x):=F(D^2u(x),0,0,x).
\]
If $\tilde{f}\in L^p(\Omega)$, then $u$ is an $L^p$-viscosity solution of 
\begin{equation}\label{eq_pendrive}
	F(D^2u,0,0,x)=\tilde{f}\hspace{.2in}\mbox{in}\hspace{.2in}\Omega.
\end{equation}
\end{Corollary}
\begin{proof}
We only prove that $u$ is an $L^p$-viscosity subsolution to \eqref{eq_pendrive}, as the case of supersolutions is analogous. Notice the proof amounts to verify the conditions in items $i.$ and $ii.$ of Lemma \ref{lem_prelvc}. 

Because $u$ solves \eqref{eq_main} in the $L^p$-viscosity sense, Lemma \ref{lem_puredequeijo} implies it is twice differentiable almost everywhere in $\Omega$. Hence, the definition of $\tilde{f}$ ensures
\[
	F(D^2u(x),0,0,x)\leq \tilde{f}(x)
\]
almost everywhere in $\Omega$, which verifies item $i.$ in Lemma \ref{lem_prelvc}. 

To address item $ii.$, we resort to Lemma \ref{lem_prelvc} in the opposite direction. Let $\phi\in W^{2,p}_{\rm loc}(\Omega)$ and suppose $x^*\in \Omega$ is a point of maximum for $u-\phi$. Since $u$ is an $L^p$-viscosity solution to \eqref{eq_main}, that lemma ensures that 
\[
	{\rm ess}\liminf_{x\to x^*}\left(\mathcal{P}^-\left(D^2(u-\phi)(x)\right)-\gamma\left|D(u-\phi)(x)\right|\right)\geq 0.
\]
Therefore, item $ii.$ also follows and the proof is complete.
\end{proof}

We also use the truncated Riesz potential of $f$. In fact, we consider its $L^p$-variant, introduced in \cite{DKM_2014}. To be precise, given $f\in L^p(\Omega)$, we define its (truncated) Riesz potential ${\bf I}^f_p(x,r)$ as
\[
	{\bf I}^f_p(x,r):=\int_0^r\left(\intav{B_\rho(x)}|f(y)|^p{\rm d}y\right)^\frac{1}{p}{\rm d}\rho.
\]
In case $p=1$ we recover the usual truncated Riesz potential. 

We proceed by stating Theorems 1.2 and 1.3 in \cite{DKM_2014}. 

\begin{Proposition}[Daskalopoulos-Kuusi-Mingione I]\label{thm_dkm12}
Let $u\in C(\Omega)$ be an $L^p$-viscosity solution to 
\[
	F(D^2u,x)=f\hspace{.2in}\mbox{in}\hspace{.2in}B_1.
\]
Suppose Assumptions A\ref{assump_F} and A\ref{assump_f} are in force. Then there exists $\theta_1$ such that, if Assumption A\ref{assump_osc} holds for $\theta\equiv\theta_1$, one has
\[
	\left|Du(x)\right|\leq C\left[{\bf I}^f_p(x,r)+\left(\intav{B_r(x)}\left|Du(y)\right|^q{\rm d}y\right)^\frac{1}{q}\right]
\]
for every $x\in \Omega$ and $r>0$ with $B_r(x)\subset \Omega$, for some universal constant $C>0$.\end{Proposition}

\begin{Proposition}[Daskalopoulos-Kuusi-Mingione II]\label{thm_dkm13}
Let $u\in C(\Omega)$ be an $L^p$-viscosity solution to 
\[
	F(D^2u,x)=f\hspace{.2in}\mbox{in}\hspace{.2in}\Omega.
\]
Suppose Assumptions A\ref{assump_F} and A\ref{assump_f} are in force. Suppose further that ${\bf I}^f_p(x,r)\to 0$ as $r\to0$, uniformly in $x$. Then there exists $\theta_2$ such that, if Assumption A\ref{assump_osc} holds for $\theta\equiv\theta_2$, $Du$ is continuous. In addition, for $\Omega'\Subset\Omega''\Subset \Omega$, and any $\delta\in (0,1]$, one has
\[
	|Du(x)-Du(y)|\leq C\left(\|Du\|_{L^\infty(\Omega'')}|x-y|^{\alpha(1-\delta)}+\sup_{z\in\left\lbrace x,y\right\rbrace}\;{\bf I}^f_p\left(z,4|x-y|^\delta\right)\right),
\]
for every $x,y\in\Omega'$, where $C=C(d,p,\lambda,\Lambda, \gamma,\omega,\Omega',\Omega'')$ and $\alpha=\alpha(d,p,\lambda,\Lambda)$.
\end{Proposition}

For the proofs of Propositions \ref{thm_dkm12} and \ref{thm_dkm13}, we refer the reader to \cite[Theorem 1.3]{DKM_2014}. We close this section by including \'Swi\k{e}ch's $W^{1,p}$-regularity result.

\begin{Proposition}[$W^{1,q}$-regularity estimates]\label{prop_andrzej}
Let $u\in C(\Omega)$ be an $L^p$-viscosity solution to \eqref{eq_main}. Suppose Assumptions A\ref{assump_F} and A\ref{assump_f} are in force. There exists $0<\overline{\theta}\ll1$ such that, if Assumption A\ref{assump_osc} holds with $\theta\equiv\overline{\theta}$, then $u\in W^{1,q}_{\rm loc}(\Omega)$ for every $1<q<p^*$, where
\[
	p^*:=\frac{pd}{d-p},\hspace{.3in}\mbox{and}\hspace{.3in}d^*=+\infty.
\]
Also, for  $\Omega'\Subset \Omega$, there exists $C=C(d,\lambda,\Lambda,\gamma,\omega,q,{\rm diam}(\Omega'),{\rm dist}(\Omega',\partial \Omega))$ such that 
\[
	\left\|u\right\|_{W^{1,q}(\Omega')}\leq C\left(\left\|u\right\|_{L^\infty(\partial\Omega)}+\left\|f\right\|_{L^p(\Omega)}\right).
\]
\end{Proposition}

The former result plays an important role in our argument since it allows us to relate the operator $F(M,p,r,x)$ with $F(M,0,0,x)$. In what follows, we detail the proofs of Theorems \ref{thm_main1} and \ref{thm_main2}.

\section{Proof of Theorems \ref{thm_main1} and \ref{thm_main2}}\label{sec_proof}

In the sequel, we detail the proofs of Theorems \ref{thm_main1} and \ref{thm_main2}. Resorting to a covering argument, we work in the unit ball $B_1$ instead of $\Omega$. As we described before, the strategy is to show that $L^p$-viscosity solutions to \eqref{eq_main} are also $L^p$-viscosity solutions to 
\[
	G(D^2u,x)=g\hspace{.2in}\mbox{in}\hspace{.2in}B_1.
\]
Then verify that $G:S(d)\times B_1\setminus\mathcal{N}\to\mathbb{R}$ and $g\in L^p(B_1)$ are in the scope of \cite{DKM_2014}. More precisely, satisfying the conditions in Theorems 1.2 and 1.3 in that paper. We continue with a proposition.

\begin{Proposition}\label{prop_wholething}
Let $u\in C(B_1)$ be an $L^p$-viscosity solution to \eqref{eq_main}. Suppose Assumptions A\ref{assump_F} and A\ref{assump_f} are in force. Suppose further that Assumption A\ref{assump_osc} holds with $\theta\equiv\overline\theta$, where $\overline\theta$ is the parameter from Proposition \ref{prop_andrzej}. Then $u$ is an $L^p$-viscosity solution for 
\[
	F(D^2u,0,0,x)=\tilde{f}\hspace{.2in}\mbox{in}\hspace{.2in}B_{9/10},
\]
where $\tilde{f}\in L^p_{\rm loc}(B_1)$ and there exists $C>0$ such that 
\[
	\big\|\tilde{f}\big\|_{L^p(B_{9/10})}\leq C\left(\left\|u\right\|_{L^\infty(B_1)}+\left\|f\right\|_{L^p(B_1)}\right).
\]
\end{Proposition}
\begin{proof}
We split the proof into two steps.

\bigskip

\noindent{\bf Step 1 -} We start by applying Proposition \ref{prop_andrzej} to the $L^p$-viscosity solutions to \eqref{eq_main}. By taking $\theta$ in Assumption A\ref{assump_osc} such that $\theta\equiv \overline{\theta}$, we get $u\in W^{1,q}_{\rm loc}(B_1)$ and
\begin{equation}\label{eq_estDu}
	\left\|Du\right\|_{L^q(B_{9/10})}\leq C\left(\left\|u\right\|_{L^\infty(\partial B_1)}+\left\|f\right\|_{L^p(B_1)}\right),
\end{equation}
for some universal constant $C>0$. Moreover, because $u$ is an $L^p$-viscosity solution to \eqref{eq_main}, Lemma \ref{lem_puredequeijo} ensures it is twice-differentiable almost everywhere in $B_1$. Define $\tilde f:B_1\to\mathbb{R}$ as
\[
	\tilde{f}(x):=F(D^2u(x),0,0,x).
\]

\bigskip

\noindent{\bf Step 2 -} Resorting once again to Lemma \ref{lem_puredequeijo}, we get that 
\[
	\tilde{f}(x)=F(D^2u(x),0,0,x)-F(D^2u(x),Du(x),u(x),x)+f(x),
\]
almost everywhere in $B_1$. Ellipticity implies 
\[
	\begin{split}
		\big|\tilde{f}(x)\big|\leq \gamma\left|Du(x)\right|+\omega\left(\left|u(x)\right|\right)+\left|f(x)\right|,
	\end{split}
\]
for almost every $x\in B_1$. Using \eqref{eq_estDu}, and noticing that one can always take $q>p$, we get $\tilde{f}\in L^p_{\rm loc}(B_1)$, with
\[
	\big\|\tilde f\big\|_{L^p(B_{9/10})}\leq C\left(\left\|u\right\|_{L^\infty(B_1)}+\left\|f\right\|_{L^p(B_1)}\right),
\]
for some universal constant $C>0$, also depending on $p$. A straightforward application of Corollary \ref{cor_salvador} completes the proof.
\end{proof}

Proposition \ref{prop_wholething} is the main ingredient leading to Theorems \ref{thm_main1} and \ref{thm_main2}. Once it is available, we proceed with the proof of those theorems.

\begin{proof}[Proof of Theorem \ref{thm_main1}]
For clarity, we split the proof into two steps.

\bigskip

\noindent{\bf Step 1 -} Because of Proposition \ref{prop_wholething}, we know that an $L^p$-viscosity solution to \eqref{eq_main} is also an $L^p$-viscosity solution to
\[
	\tilde{F}(D^2u,x)=\tilde{f}\hspace{.2in}\mbox{in}\hspace{.2in}B_{9/10},
\]
where 
\[
	\tilde{F}(M,x):=F(M,0,0,x),
\]
and $\tilde{f}$ is defined as in Proposition \ref{prop_wholething}. To conclude the proof, we must ensure that $\tilde F$ satisfies the conditions in Proposition \ref{thm_dkm12}.

\bigskip

\noindent{\bf Step 2 -} One easily verifies that $\tilde F$ satisfies a $(\lambda,\Lambda)$-ellipticity condition, inherited from the original operator $F$. It remains to control the oscillation of $\tilde F(M,x)$ vis-a-vis its fixed-coefficient counterpart, $\tilde F(M,x_0)$, for $x_0\in B_{9/10}$.

Because 
\[
	\tilde F(M,x)-\tilde F(M,x_0)=F(M,0,0,x)-F(M,0,0,x_0),
\]
one may take $\theta\equiv\theta_1$ in Assumption \ref{assump_osc} to ensure that $\tilde F$ satisfies the conditions in Proposition \ref{thm_dkm12}. Taking
\[
	\theta^*:=\min\left(\, \theta_1,\;\overline\theta\,\right)
\]
and applying Proposition \ref{thm_dkm12} to $u$, the proof is complete.
\end{proof}

The proof of Theorem \ref{thm_main2} follows word for word the previous one, except for the choice of $\theta^*:=\min\left(\, \theta_2,\;\overline\theta\,\right)$, and is omitted.

\bigskip

\noindent{\bf Acknowledgements.} This work was partially supported by the Centre for Mathematics of the University of Coimbra - UIDB/00324/2020, funded by the Portuguese Government through FCT/MCTES. EP is partially funded by FAPERJ (\# E-26/201.390/2021). MW is partially supported by FAPERJ (\# E-26/201.647/2021).

\bibliography{biblio}

\begin{thebibliography}{10}

\bibitem{BG1989}
Lucio Boccardo and Thierry Gallou\"{e}t.
\newblock Nonlinear elliptic and parabolic equations involving measure data.
\newblock {\em J. Funct. Anal.}, 87(1):149--169, 1989.

\bibitem{BG1992}
Lucio Boccardo and Thierry Gallou\"{e}t.
\newblock Nonlinear elliptic equations with right-hand side measures.
\newblock {\em Comm. Partial Differential Equations}, 17(3-4):641--655, 1992.

\bibitem{Caffarelli_1989}
Luis~A. Caffarelli.
\newblock Interior a priori estimates for solutions of fully nonlinear
  equations.
\newblock {\em Ann. of Math. (2)}, 130(1):189--213, 1989.

\bibitem{CC}
Luis~A. Caffarelli and Xavier Cabr\'{e}.
\newblock {\em Fully nonlinear elliptic equations}, volume~43 of {\em American
  Mathematical Society Colloquium Publications}.
\newblock American Mathematical Society, Providence, RI, 1995.

\bibitem{CCKS_1996}
Luis~A. Caffarelli, Michael~G. Crandall, Maciej Kocan, and Andrzej
  \'Swi\k{e}ch.
\newblock On viscosity solutions of fully nonlinear equations with measurable
  ingredients.
\newblock {\em Comm. Pure Appl. Math.}, 49(4):365--397, 1996.

\bibitem{C2011}
Andrea Cianchi.
\newblock Nonlinear potentials, local solutions to elliptic equations and
  rearrangements.
\newblock {\em Ann. Sc. Norm. Super. Pisa Cl. Sci. (5)}, 10(2):335--361, 2011.

\bibitem{CEL}
Michael~G. Crandall, Lawrence~C. Evans, and Pierre-Louis Lions.
\newblock Some properties of viscosity solutions of {H}amilton-{J}acobi
  equations.
\newblock {\em Trans. Amer. Math. Soc.}, 282(2):487--502, 1984.

\bibitem{CIL}
Michael~G. Crandall, Hitoshi Ishii, and Pierre-Louis Lions.
\newblock User's guide to viscosity solutions of second order partial
  differential equations.
\newblock {\em Bull. Amer. Math. Soc. (N.S.)}, 27(1):1--67, 1992.

\bibitem{CL}
Michael~G. Crandall and Pierre-Louis Lions.
\newblock Viscosity solutions of {H}amilton-{J}acobi equations.
\newblock {\em Trans. Amer. Math. Soc.}, 277(1):1--42, 1983.

\bibitem{DKM_2014}
Panagiota Daskalopoulos, Tuomo Kuusi, and Giuseppe Mingione.
\newblock Borderline estimates for fully nonlinear elliptic equations.
\newblock {\em Comm. Partial Differential Equations}, 39(3):574--590, 2014.

\bibitem{CDF2022}
Cristiana De~Filippis.
\newblock Quasiconvexity and partial regularity via nonlinear potentials.
\newblock {\em J. Math. Pures Appl. (9)}, 163:11--82, 2022.

\bibitem{FM2021}
Cristiana De~Filippis and Giuseppe Mingione.
\newblock Lipschitz bounds and nonautonomous integrals.
\newblock {\em Arch. Ration. Mech. Anal.}, 242(2):973--1057, 2021.

\bibitem{FM2022}
Cristiana De~Filippis and Giuseppe Mingione.
\newblock Nonuniformly elliptic schauder theory.
\newblock {\em {\rm arXiv:2201.07369 [math.AP]}}, 2022.

\bibitem{FP2022}
Cristiana De~Filippis and Mirco Piccinini.
\newblock Borderline global regularity for nonuniformly elliptic systems.
\newblock {\em {\rm arXiv:2206.15330 [math.AP]}}, 2022.

\bibitem{FS2022}
Cristiana De~Filippis and Bianca Stroffolini.
\newblock Singular multiple integrals and nonlinear potentials.
\newblock {\em {\rm arXiv:2203.05519 [math.AP]}}, 2022.

\bibitem{DM2009}
Frank Duzaar and Giuseppe Mingione.
\newblock Partial differential equations---gradient estimates in non-linear
  potential theory.
\newblock {\em Atti Accad. Naz. Lincei Rend. Lincei Mat. Appl.},
  20(2):179--190, 2009.

\bibitem{DM2010}
Frank Duzaar and Giuseppe Mingione.
\newblock Gradient continuity estimates.
\newblock {\em Calc. Var. Partial Differential Equations}, 39(3-4):379--418,
  2010.

\bibitem{DM2010a}
Frank Duzaar and Giuseppe Mingione.
\newblock Gradient estimates via linear and nonlinear potentials.
\newblock {\em J. Funct. Anal.}, 259(11):2961--2998, 2010.

\bibitem{DM2011}
Frank Duzaar and Giuseppe Mingione.
\newblock Gradient estimates via non-linear potentials.
\newblock {\em Amer. J. Math.}, 133(4):1093--1149, 2011.

\bibitem{Escauriaza}
Luis Escauriaza.
\newblock {$W^{2,n}$} a priori estimates for solutions to fully nonlinear
  equations.
\newblock {\em Indiana Univ. Math. J.}, 42(2):413--423, 1993.

\bibitem{Evans1}
Lawrence~C. Evans.
\newblock Classical solutions of fully nonlinear, convex, second-order elliptic
  equations.
\newblock {\em Comm. Pure Appl. Math.}, 35(3):333--363, 1982.

\bibitem{FS1984}
Eugene~B. Fabes and Daniel~W. Stroock.
\newblock The {$L^p$}-integrability of {G}reen's functions and fundamental
  solutions for elliptic and parabolic equations.
\newblock {\em Duke Math. J.}, 51(4):997--1016, 1984.

\bibitem{KM1992}
Tero Kilpel\"{a}inen and Jan Mal\'{y}.
\newblock Degenerate elliptic equations with measure data and nonlinear
  potentials.
\newblock {\em Ann. Scuola Norm. Sup. Pisa Cl. Sci. (4)}, 19(4):591--613, 1992.

\bibitem{Krylov1}
Nicolai~V. Krylov.
\newblock Boundedly inhomogeneous elliptic and parabolic equations.
\newblock {\em Izv. Akad. Nauk SSSR Ser. Mat.}, 46(3):487--523, 670, 1982.

\bibitem{KrSa1}
Nicolai~V. Krylov and Mikhail~V. Safonov.
\newblock A property of the solutions of parabolic equations with measurable
  coefficients.
\newblock {\em Izv. Akad. Nauk SSSR Ser. Mat.}, 44(1):161--175, 239, 1980.

\bibitem{KM2011}
Tuomo Kuusi and Giuseppe Mingione.
\newblock A surprising linear type estimate for nonlinear elliptic equations.
\newblock {\em C. R. Math. Acad. Sci. Paris}, 349(15-16):889--892, 2011.

\bibitem{KM2012}
Tuomo Kuusi and Giuseppe Mingione.
\newblock Pointwise gradient estimates.
\newblock {\em Nonlinear Anal.}, 75(12):4650--4663, 2012.

\bibitem{KM2012a}
Tuomo Kuusi and Giuseppe Mingione.
\newblock Potential estimates and gradient boundedness for nonlinear parabolic
  systems.
\newblock {\em Rev. Mat. Iberoam.}, 28(2):535--576, 2012.

\bibitem{KM2013}
Tuomo Kuusi and Giuseppe Mingione.
\newblock Gradient regularity for nonlinear parabolic equations.
\newblock {\em Ann. Sc. Norm. Super. Pisa Cl. Sci. (5)}, 12(4):755--822, 2013.

\bibitem{KM2013a}
Tuomo Kuusi and Giuseppe Mingione.
\newblock Linear potentials in nonlinear potential theory.
\newblock {\em Arch. Ration. Mech. Anal.}, 207(1):215--246, 2013.

\bibitem{KM2014}
Tuomo Kuusi and Giuseppe Mingione.
\newblock Borderline gradient continuity for nonlinear parabolic systems.
\newblock {\em Math. Ann.}, 360(3-4):937--993, 2014.

\bibitem{KM2014c}
Tuomo Kuusi and Giuseppe Mingione.
\newblock Guide to nonlinear potential estimates.
\newblock {\em Bull. Math. Sci.}, 4(1):1--82, 2014.

\bibitem{KM2014b}
Tuomo Kuusi and Giuseppe Mingione.
\newblock Riesz potentials and nonlinear parabolic equations.
\newblock {\em Arch. Ration. Mech. Anal.}, 212(3):727--780, 2014.

\bibitem{KM2014a}
Tuomo Kuusi and Giuseppe Mingione.
\newblock The {W}olff gradient bound for degenerate parabolic equations.
\newblock {\em J. Eur. Math. Soc. (JEMS)}, 16(4):835--892, 2014.

\bibitem{KM2016a}
Tuomo Kuusi and Giuseppe Mingione.
\newblock Nonlinear potential theory of elliptic systems.
\newblock {\em Nonlinear Anal.}, 138:277--299, 2016.

\bibitem{KM2016}
Tuomo Kuusi and Giuseppe Mingione.
\newblock Partial regularity and potentials.
\newblock {\em J. \'{E}c. polytech. Math.}, 3:309--363, 2016.

\bibitem{KM2018}
Tuomo Kuusi and Giuseppe Mingione.
\newblock Vectorial nonlinear potential theory.
\newblock {\em J. Eur. Math. Soc. (JEMS)}, 20(4):929--1004, 2018.

\bibitem{M2011a}
Giuseppe Mingione.
\newblock Gradient potential estimates.
\newblock {\em J. Eur. Math. Soc. (JEMS)}, 13(2):459--486, 2011.

\bibitem{M2011}
Giuseppe Mingione.
\newblock Nonlinear measure data problems.
\newblock {\em Milan J. Math.}, 79(2):429--496, 2011.

\bibitem{M2014}
Giuseppe Mingione.
\newblock Recent advances in nonlinear potential theory.
\newblock In {\em Trends in contemporary mathematics}, volume~8 of {\em
  Springer INdAM Ser.}, pages 277--292. Springer, Cham, 2014.

\bibitem{M2018}
Giuseppe Mingione.
\newblock Recent progress in nonlinear potential theory.
\newblock In {\em European {C}ongress of {M}athematics}, pages 501--524. Eur.
  Math. Soc., Z\"{u}rich, 2018.

\bibitem{Swiech_1997}
Andrzej \'{S}wi\k{e}ch.
\newblock {$W^{1,p}$}-interior estimates for solutions of fully nonlinear,
  uniformly elliptic equations.
\newblock {\em Adv. Differential Equations}, 2(6):1005--1027, 1997.

\bibitem{Trudinger}
Neil~S. Trudinger.
\newblock H\"{o}lder gradient estimates for fully nonlinear elliptic equations.
\newblock {\em Proc. Roy. Soc. Edinburgh Sect. A}, 108(1-2):57--65, 1988.

\bibitem{TW2002}
Neil~S. Trudinger and Xu-Jia Wang.
\newblock On the weak continuity of elliptic operators and applications to
  potential theory.
\newblock {\em Amer. J. Math.}, 124(2):369--410, 2002.

\bibitem{Winter_2009}
Niki Winter.
\newblock {$W^{2,p}$} and {$W^{1,p}$}-estimates at the boundary for solutions
  of fully nonlinear, uniformly elliptic equations.
\newblock {\em Z. Anal. Anwend.}, 28(2):129--164, 2009.

\end{thebibliography}
\bibliographystyle{plain}

\bigskip

\noindent\textsc{Edgard A. Pimentel \Letter}\\University of Coimbra\\CMUC, Department of Mathematics\\ 3001-501 Coimbra, Portugal\\and\\Pontifical Catholic University of Rio de Janeiro -- PUC-Rio\\22451-900, G\'avea, Rio de Janeiro-RJ, Brazil\\\noindent\texttt{edgard.pimentel@mat.uc.pt}

\bigskip

\noindent\textsc{Miguel Walker}\\Pontifical Catholic University of Rio de Janeiro -- PUC-Rio\\22451-900, G\'avea, Rio de Janeiro-RJ, Brazil\\\noindent\texttt{mwalkeru@mat.puc-rio.br}

\end{document}